\documentclass[a4paper]{amsart}
\usepackage[utf8]{inputenc}
\usepackage[pdfusetitle]{hyperref}

\usepackage{amsmath,amsthm,stackengine,scalerel}
\usepackage{amssymb}
\usepackage{mathrsfs}
\usepackage{mathtools}
\usepackage{tikz-cd}
\usepackage{tikz}
\usetikzlibrary{arrows,calc,decorations.pathmorphing,shapes,fit}
\usepackage{enumitem}
\setlist[enumerate,1]{label={\normalfont{(\arabic*)}}}

\usepackage[capitalise]{cleveref}
\usepackage{csquotes}

\usepackage{todonotes}

\usepackage{faktor}

\usepackage{comment}

\newtheoremstyle{mystyle}
{}
{}
{\normalfont}
{}
{\bfseries}
{}
{0.5 em}
{\thmname{#1}\thmnumber{ #2}\normalfont\thmnote{ (#3)}\bfseries.}

\newtheoremstyle{mystyle2}
{}
{}
{\itshape}
{}
{\bfseries}
{}
{0.5 em}
{\thmname{#1}\thmnumber{ #2}\normalfont\thmnote{ (#3)}\bfseries.}

\theoremstyle{mystyle2}
\newtheorem{thm}{Theorem}[section]
\newtheorem*{thm*}{Theorem}
\newtheorem{prop}[thm]{Proposition}
\crefname{prop}{Proposition}{Propositions}
\newtheorem{lem}[thm]{Lemma}
\newtheorem*{lem*}{Lemma}
\newtheorem{kor}[thm]{Corollary}

\theoremstyle{mystyle}
\newtheorem{defi}[thm]{Definition}
\newtheorem*{defi*}{Definition}
\crefname{defi}{Definition}{Definitions}
\newtheorem{examp}[thm]{Example}
\crefname{examp}{Example}{Examples}
\newtheorem*{examp*}{Example}
\newtheorem{constr}[thm]{Construction}
\newtheorem*{constr*}{Construction}
\newtheorem{rem}[thm]{Remark}

\newtheorem{fact}[thm]{Fact}
\newtheorem{prob}[thm]{Problem}

\DeclareMathOperator{\chara}{char}

\DeclareMathOperator{\res}{res}

\newcommand{\sort}[1]{\mathbf{#1}}
\newcommand{\map}[1]{\mathrm{#1}}

\newcommand{\Z}{\mathbb{Z}}
\newcommand{\N}{\mathbb{N}}
\newcommand{\F}{\mathbb{F}}
\newcommand{\Q}{\mathbb{Q}}

\DeclareMathOperator{\Th}{Th}

\newcommand{\x}{^\times}

\newcommand{\U}{\mathcal{U}}

\newcommand{\C}{\mathcal{C}}

\renewcommand{\L}{\mathcal{L}}
\newcommand{\Lring}{\L_{\mathrm{ring}}}
\newcommand{\Loag}{\L_{\mathrm{oag}}}
\newcommand{\Lval}{\L_{\mathrm{val}}}
\newcommand{\Lgammak}{\L_{\Gamma k}}

\renewcommand{\O}{\mathcal{O}}

\newcommand{\barv}{\overline{v}}
\newcommand{\barw}{\overline{w}}

\newcommand{\CAKEeq}[0]{{\normalfont{CAKE}\({}^\equiv\)}}
\newcommand{\AKEeq}[0]{{\normalfont{AKE}\({}^\equiv\)}}

\makeatletter
\let\@wraptoccontribs\wraptoccontribs
\makeatother

\title[Composition AKE principles and tame fields of mixed characteristic]{Composition Ax--Kochen/Ershov principles and tame fields of mixed characteristic}
\author{Margarete Ketelsen}
\contrib[with an appendix by]{Philip Dittmann}
\address{Margarete Ketelsen, Institut für Mathematische Logik und Grundlagenforschung, Universität Münster, Einsteinstraße 61, 48149 Münster, Germany}
\email{margarete.ketelsen@uni-muenster.de}
\address{Philip Dittmann, Department of Mathematics, University of Manchester, Manchester M13 9PL, UK}
\email{philip.dittmann@manchester.ac.uk}

\begin{document}
	\maketitle
	
	\begin{abstract}
		We study in which settings we have a composition AKE principle, i.e. when the theory of the coarsening \((K,w)\) and the theory of the induced valuation \((Kw,\overline{v})\) determine the theory of the composition \((K,v)\).
		We show that this is the case when \((K,w)\) is tame of equal characteristic, and provide counterexamples in mixed characteristic.
		We further show that, for a tame field of mixed characteristic, the theory of the valued field cannot, in general, be determined solely by the theories of its underlying field, its residue field, and its value group.
	\end{abstract}
	

	\section{Introduction}
	First proved by Ax and Kochen and independently Ershov in the 1960s, an Ax--Kochen/Ershov (AKE) principle says that certain questions about valued fields can be reduced to the model theory of their residue fields and value groups, see \cite{ax-kochen1965diophantine,ershov1965elementary} for their original work covering henselian fields of equicharacteristic zero and \(\Q_p\).

	In this paper we want to understand the first-order theory of valued fields using a related approach:
	Instead of reducing to the residue field and the value group, we can \emph{decompose} a valued field into potentially easier to understand valuations of lower rank: a coarsening and a valuation induced on the residue field of the coarsening.
	
	Let \((K,v)\) be a valued field. 
	A valuation \(v'\) on \(K\) is called a \emph{coarsening} of \(v\) if \(\O_{v'}\supseteq\O_v\).
	Given a valued field \((K,v)\) and a coarsening \(v'\), we define the induced valuation \(\overline{v}\) of \(v\) on the residue field \(Kv'\) of \(v'\) to be the valuation with valuation ring
	\[
	\O_{\overline{v}}\coloneqq \res_{v'}(\O_v)\subseteq Kv'.
	\]
	We write \(v=\overline{v}\circ v'\) and say \(v\) \emph{decomposes} into \(\overline{v}\) and \(v'\).
	
	Conversely, we can also consider \emph{compositions} of valuations: 
	Given a valued field \((K,w)\) and a valuation \(u\) on its residue field \(Kw\) we can define \(v=u\circ w\) by letting
	\[
	\O_v\coloneqq \res_w^{-1}(\O_u).
	\]
	
	We consider valued fields as structures in the \emph{language of valued fields} \(\Lval=\Lring \cup \{\O\}\), where \(\Lring=\{0,1,+,-,\cdot\}\) is the language of rings and \(\O\) is a unary relation symbol, for the valuation ring.
	
	The main goal of this article is to investigate the following question:
	
	\begin{prob}
		\label{prob:cake}
		Let \((K,v)\) be a henselian valued field and let \(\barv\) be a valuation on \(Kv\).
		
		Is \(\Th(K,\barv\circ v)\) determined by \(\Th(K,v)\) and \(\Th(Kv,\barv)\)?
	\end{prob}
	This is also an AKE-type question as we again can reduce to the model theory of simpler objects. Here, these are the coarsening and the induced valuation.
	
	If the answer to \cref{prob:cake} is \enquote{yes}, we will say that the valued field \((K,\barv\circ v)\) with the given decomposition \(\barv\circ v\) has the composition \AKEeq{}-property (or \CAKEeq-property for short).
	However, unsurprisingly, often it suffices that the first valued field \((K,v)\) is nice enough (i.e. admits certain AKE properties). 
	Thus, we might also say that the valued field \((K,v)\) has the \CAKEeq{}-property, if for any \(\barv\) on the residue field \(Kv\), the valued field \((K,\barv\circ v)\) with the given decomposition \(\barv \circ v\) has the \CAKEeq{}-property.

	In particular, we want to study the \CAKEeq{}-property for tame valued fields. 
	A valued field \((K,v)\) is called \emph{tame} if it is algebraically maximal, its residue field \(Kv\) is perfect, and if \(\chara(Kv)=p>0\), the value group \(vK\) is \(p\)-divisible. 
	Tame fields have been introduced and studied by Franz-Viktor Kuhlmann in \cite{kuhlmann2016algebra}. They are well understood model-theoretically: tame fields of equal characteristic admit \AKEeq{}, while in mixed characteristic they at least still have \normalfont{AKE}\({}^\preceq\).
	
	We will see that the \CAKEeq{}-property holds for tame fields of equal characteristic, but fails in general for tame fields of mixed characteristic.

	In addition, we show that, in general, for tame fields of mixed characteristic, the \(\Lval\)-theory of \((K,v)\) is \textbf{not} determined by
	\begin{itemize}
		\item the \(\Lring\)-theory of \(K\),
		\item the \(\Lring\)-theory of \(Kv\), and
		\item the \(\Loag\)-theory of \(vK\).
	\end{itemize}
	This provides new counterexamples for tame \AKEeq{} in mixed characteristic with \textbf{the same underlying field}.
	Specifically, the algebraic part is not responsible for the failure of the \AKEeq{}-principle.
	
	In our counterexamples, the \AKEeq{}-principle fails because of the pointed value groups \((vK,v(p))\) are not elementarily equivalent, see \cref{rem:pointed-val-gp}.
	Philip Dittmann constructed other examples, where \AKEeq{} fails, but the pointed value groups, and the algebraic parts coincide, see \cref{appendix}.

	\section{Resplendent AKE and CAKE}
	
	It is well known, that the \CAKEeq{}-principle holds for henselian fields of equicharacteristic zero.
	This was made explicit in \cite[Lemma~6.5]{anscombe2024characterizing}, although the proof given there is not entirely correct (see \cref{rem:stable-embeddedness-is-a-trap}).
	
	In fact, \CAKEeq{} can also be seen as a consequence of \emph{resplendent \AKEeq{}}.
	
	Let \(\Lgammak=(\Lring,\Loag\cup\{\infty\},\Lring,\map{v},\map{res})\) be the \(3\)-sorted language of valued fields with sorts \(\sort{K}\) endowed with \(\Lring\), \(\sort{\Gamma}\) endowed with \(\Loag\cup \{\infty\}\) and \(\sort{k}\) endowed with \(\Lring\), and maps \(\map{v}\colon \sort{K}\to\sort{\Gamma}\) and \(\map{res}\colon \sort{K}\to \sort{k}\).
	
	\begin{defi}[Resplendent \AKEeq{}]
		We say a class of valued fields \(\C\) has \emph{resplendent \AKEeq{}} if for any \(\{\sort{\Gamma},\sort{k}\}\)-enrichment (see {\cite[Definition~A.2]{rideau2017some}}) \(\L\) of \(\Lgammak\), and \(\L\)-structures \((K,v,\ldots)\), \((L,w,\ldots)\) such that their reducts \((K,v)\) and \((L,w)\) down to \(\Lgammak\) are in \(\C\), we have
		\begin{align*}
			(Kv,\ldots)&\equiv (Lw,\ldots) \text{ in } \L|_\sort{k}
			\text{ and }
			(vK,\ldots)\equiv (wL,\ldots) \text{ in } \L|_\sort{\Gamma}\\
			&\quad \Longrightarrow \quad
			(K,v,\ldots)\equiv (L,w,\ldots) \text{ in } \L.
		\end{align*}
	\end{defi}
	
	For the class of equicharacteristic zero henselian valued fields, resplendent \AKEeq{} follows from a resplendent version of quantifier elimination in the Denef--Pas language, see \cite{pas1989uniform}.
	
	\begin{thm}[{\cite[Section~7.2]{vdD2014lectures}}]
		The class of equicharacteristic zero henselian valued fields admits resplendent \AKEeq{}.
	\end{thm}
	
	\begin{prop}[\CAKEeq{} from resplendent \AKEeq{}]
		\label{prop:cake-from-rake}
		Let \((L,w)\) and \((F,u)\) be valued fields with decompositions \(w=\barw\circ w'\) and \(u=\overline{u}\circ u'\).
		Suppose that \((L,w')\) and \((F,u')\) have resplendent \AKEeq{}.
		
		Then, 
		\[
		(L,w')\equiv (F,u') \quad\text{and}\quad (Lw',\barw)\equiv(Fu',\overline{u})\quad\Longrightarrow\quad (L,w)\equiv (F,u)  \text{ in } \Lval.
		\]
	\end{prop}
	\begin{proof}
		Let \(\L\coloneqq \Lgammak\cup\{\mathsf{O}\}\), where \(\mathsf{O}\) is a unary relation symbol on the residue field sort \(\sort{k}\), thus \(\L\) is a \(\{\sort{k}\}\)-enrichment of \(\Lgammak\).
		Note that \(\L_{|_\sort{k}}=\Lring\cup\{\mathsf{O}\}=\Lval\).
		The decomposed valued fields \((L,w',\overline{w})\) and \((F,u',\overline{u})\) are \(\L\)-structures, when considering the \(\Lgammak\)-structure of \((L,w')\) and \((F,u')\), and interpreting \(\mathsf{O}\) as the valuation ring of \(\overline{w}\) or \(\overline{u}\), respectively.
		
		Now, assume that \((L,w')\equiv(F,u')\) and \((Lw',\overline w)\equiv(Fu',\overline u)\) in \(\Lval\). As the \(\Loag\)-structure of the value group is interpretable in the \(\Lval\)-structure of the valued field, we get \(w'L\equiv u'F\) in \(\Loag\) and \((Lw',\overline w)\equiv(Fu',\overline u)\) in \(\Lval\).
		By resplendent \AKEeq{}, we get \((L,w',\overline{w})\equiv (F,u',\overline u)\) in \(\L\).
		
		Note that the \(\Lval\)-structure of the composed valued field \((L,w)\) is interpretable in the \(\L\)-structure \((L,w',\barw)\), as 
		\[
		x\in \O_w \quad \Longleftrightarrow \quad x\in \O_{w'} \text{ and } \res_{w'}(x)\in \O_{\overline{w}},
		\]
		and the same holds for \((F,u)\) and \((F,u',\overline u)\).
		Thus it follows that \((L,w)\equiv(F,u)\).
	\end{proof}
	
	\begin{rem}
		Note that in the proof of \cref{prop:cake-from-rake}, we only used \AKEeq{} for a \(\{\sort{k}\}\)-enrichment of \(\Lgammak\). 
		Thus, it would be enough to assume resplendent \AKEeq{} with respect to the residue field.
	\end{rem}
	
	We finish this section by proving resplendent \AKEeq{} for tame fields of equal characteristic. 
	This is a consequence of the relative embedding property for tame fields.
	\begin{defi}[{\cite[Section~6]{kuhlmann2016algebra}}]
		Let \(\C\) be a class of valued fields. 
		We say that \(\C\) has the \emph{relative embedding property} if whenever we have \((L,v),(K^*,v^*)\in\C\) with a common substructure \((K,v)\) such that 
		\begin{itemize}
			\item \((K,v)\) is defectless,
			\item \((K^*,v^*)\) is \(|L|^+\)-saturated,
			\item \(vL/vK\) is torsion free and \(Lv|Kv\) is separable, and
			\item there are embeddings \(\rho\colon vL\to v^*K^*\) over \(vK\) and \(\sigma\colon Lv\to K^*v^*\) over \(Kv\),
		\end{itemize}
		then there exists an embedding \(\iota\colon (L,v)\to (K^*,v^*)\) over \(K\) which respects \(\rho\) and \(\sigma\).
	\end{defi}
	\begin{examp}[{\cite[Theorem~7.1]{kuhlmann2016algebra}}]
		\label{ex:REP}
		The elementary class of tame fields has the relative embedding property.
	\end{examp}
	One has to repeat the proof steps for proving \AKEeq{} as in \cite[Lemma~6.1]{kuhlmann2016algebra} while paying some attention to the extra structure on the residue field and the value group.
	For the convenience of the reader we will sketch the proof.
	
	\begin{lem}[Resplendent relative subcompleteness from the relative embedding property]
		\label{lem:resplendent-relative-subcompleteness}
		Let \(\L\supseteq\Lgammak\) be a \(\{\sort{\Gamma},\sort{k}\}\)-enrichment of \(\Lgammak\). 
		Let \((L,w,\ldots)\) and \((F,u,\ldots)\) be \(\L\)-structures such that the reducts \((L,w)\) and \((F,u)\) to \(\Lgammak\) belong to an elementary class of defectless valued fields that has the relative embedding property. 
		Let \((K,v,\ldots)\) be a common substructure of \((L,w,\ldots)\) and \((F,u,\ldots)\) and assume that \((K,v)\) is defectless, \(wL/vK\) is torsion-free and \(Lw|Kv\) is separable.
		Then
		\begin{align*}
			(Lw,\ldots)&\equiv_{Kv} (Fu,\ldots) \text{ in } \L_{|_\sort{k}} \text{ and } (wL,\ldots)\equiv_{vK} (uF,\ldots) \text{ in } \L|_\sort{\Gamma}
			\\\quad &\Longrightarrow \quad (L,w,\ldots)\equiv_{K} (F,u,\ldots) \text{ in }\L.
		\end{align*}
	\end{lem}
	\begin{proof}
		The proof is essentially the same as \cite[Proof of Lemma~6.1]{kuhlmann2016algebra} while paying some attention to the extra structure. We comment on the places where extra care is needed, while sketching out the entire proof.
		
		By the Keisler--Shelah isomorphism theorem \cite[Theorem~2.5.36]{marker2006model}, there are ultrapowers and isomorphisms
		\begin{align*}
			\rho_0\colon (wL,\ldots)^\U &\overset{\cong}{\longrightarrow} (uF,\ldots)^\U\text{ over \(vK\)}\\
			\sigma_0\colon (Lw,\ldots )^\U &\overset{\cong}{\longrightarrow} (Fu,\ldots)^\U\text{ over \(Kv\)}
		\end{align*}
		When ignoring the extra structure, \(\rho_0\) and \(\sigma_0\) can also be seen as isomorphisms of the underlying ordered abelian groups or fields, respectively.
		
		Now we do a back and forth construction.
		We construct chains \(((L_i,w,\ldots))_{i\in \N}\), \(((F_i,u,\ldots))_{i\in \N}\) by taking ultrapowers consecutively, choosing the ultrafilters in a way that we get sufficient saturation, and such that with the Keisler--Shelah isomorphism theorem \cite[Theorem~2.5.36]{marker2006model}, we get isomorphisms \(\rho_i\) and \(\sigma_i\) of the enriched value groups and residue fields, respectively, on each level. 
		In particular, again ignoring the extra structure, the \(\rho_i\) are isomorphisms of ordered abelian groups and the \(\sigma_i\) are isomorphisms of fields.
		
		Now, as in \cite[Proof of Lemma~6.1]{kuhlmann2016algebra} and using the relative embedding property, we inductively construct embeddings \(\iota_i\) (for \(i\) even) and \(\iota_i'\) (for \(i\) odd) of valued fields (for now without enrichment), inducing \(\rho_i\) and \(\sigma_i\), or in the case of the \(\iota_i'\), their inverses.
		Note, the \(\iota_i\) and \(\iota_i'\) are also embeddings of the enriched valued fields, because \(\rho_i\) and \(\sigma_i\) (and their inverses) were isomorphisms even with the extra structure on value group and residue field. 
		See \cref{fig:back-and-forth-drawing} for a drawing of the constructed embeddings.
		
		\begin{figure}[htb]
			\centering
			\begin{tikzcd}
				{(L^*,w,\ldots)} \arrow[rr, "\cong"]                                  &                                                    & {(F^*,u,\ldots)}                                                  \\
				\ {\phantom{(F_4,u,\ldots)}} &  & \ {\phantom{(F_4,u,\ldots)}} \\
				\ & & {\ {\phantom{(F_4,u,\ldots)}}} \arrow[llu,dotted,hook]\\
				{(L_4,w,\ldots)} \arrow[uu, no head, dotted] \arrow[rru, dotted, hook] &                                                    & {(F_4,u,\ldots)} \arrow[uu, no head, dotted]                       \\
				{(L_3,w,\ldots)} \arrow[u, no head]                                   &                                                    & {(F_3,u,\ldots)} \arrow[u, no head] \arrow[llu, "\iota'_3", hook] \\
				{(L_2,w,\ldots)} \arrow[u, no head] \arrow[rru, "\iota_2", hook]      &                                                    & {(F_2,u,\ldots)} \arrow[u, no head]                               \\
				{(L_1,w,\ldots)} \arrow[u, no head]                                   &                                                    & {(F_1,u,\ldots)} \arrow[u, no head] \arrow[llu, "\iota'_1", hook] \\
				{(L_0,w,\ldots)} \arrow[u, no head] \arrow[rru, "\iota_0", hook]      &                                                    & {(F_0,u,\ldots)} \arrow[u, no head]                               \\
				& {(K,v,\ldots)} \arrow[lu, no head] \arrow[ru, no head] &                                                              
			\end{tikzcd}
			\caption{Back and forth construction for the proof of \cref{lem:resplendent-relative-subcompleteness}}
			\label{fig:back-and-forth-drawing}
		\end{figure}
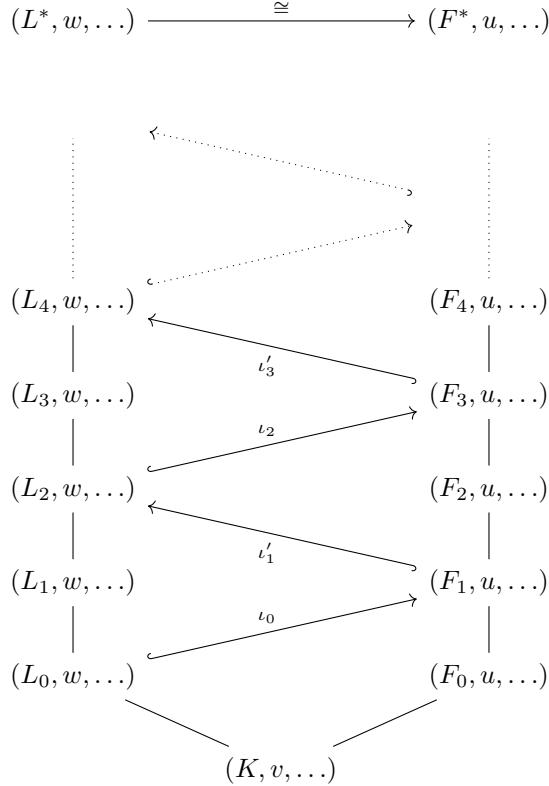
		
		Then, as in \cite[Proof of Lemma~6.1]{kuhlmann2016algebra}, we take the (set-theoretic) union of all the \(\iota_i\) (for \(i\) even) and obtain an \(\L\)-isomorphism over \((K,v,\ldots)\) between \((L^\ast,w^\ast,\ldots)\) and \((F^\ast,u^\ast,\ldots)\), which are elementary extensions of \((L,w,\ldots)\) and \((F,u,\ldots)\), respectively.
		Thus, \((L,w,\ldots)\equiv_K(F,u,\ldots)\) in \(\L\).
	\end{proof}
	
	As a corollary, we get resplendent \AKEeq{} for tame fields of equal characteristic.
	
	\begin{kor}
		Let \(\C\) be an elementary class of defectless fields of \textbf{equal} characteristic that has the relative embedding property. 
		Then \(\C\) has resplendent \AKEeq{}.
		In particular, \(\C\) has the \CAKEeq{}-property.
	\end{kor}

	\section{Counterexamples in mixed characteristic}
	
	\subsection{Lifting automorphisms}
	
	To highlight what can go wrong, we start with a proof sketch of \CAKEeq{} that has a \textbf{missing step}, which will serve as a motivation for counterexamples.
	
	\begin{rem}[Apparent proof of \CAKEeq{}]
		\label{rem:wrong-proof-CAKE}
		Let \((K,v)\) and \((L,w)\) be valued fields with decompositions \(v=\barv\circ v'\) and \(w=\barw\circ w'\).
		Assume that \((K,v')\equiv (L,w')\) and \((Kv',\barv)\equiv (Lw',\barw)\) in \(\Lval\). We want to prove that \((K,v)\equiv(L,w)\).
		
		By the Keisler--Shelah isomorphism theorem \cite[Theorem~2.5.36]{marker2006model} there is an index set \(I\) and an ultrafilter \(\U\) on \(I\) such that \((K,v')^\U\cong (L,w')^\U\) and \((Kv',\barv)^\U\cong (Lw',\barw)^\U\).
		One might think that this already implies that \((K,v)^\U\cong (L,w)^\U\) which would mean we are done.
		However, this is generally not the case since the isomorphisms do not need to be compatible (and we will see very concrete counterexamples later, see Examples \ref{ex:fully-tame-CAKE-counterexample} and \ref{ex:first-CAKE-counterex}).
		
		We can see what is going on after naming the isomorphisms:
		\[
		\psi\colon (K,v')^\U\overset{\cong}{\longrightarrow} (L,w')^\U
		\text{ and } 
		\phi\colon(Kv',\barv)^\U\overset{\cong}{\longrightarrow} (Lw',\barw)^\U.
		\]
		Now, \(\psi\) induces an isomorphism \(\overline{\psi}\) on the residue fields.
		Then \(\phi^{-1}\circ\overline{\psi}\) is an automorphism of \((Kv')^\U\).
		
		If we can somehow \textbf{lift this automorphism} to an automorphism \(\chi\) of the valued field \((K,v')^\U\), we are done:
		
		Then, \(\chi\circ \psi^{-1}\) is an isomorphism of valued fields inducing \((\phi^{-1}\circ\overline{\psi})\circ\overline{\psi}^{-1}=\phi^{-1}\).
		Since \(\phi^{-1}\colon(Lw',\barw)^\U \overset{\cong}{\longrightarrow} (Kv',\barv)^\U\) was also an isomorphism between the valued fields induced by \(w\) and \(v\), we have that \(\chi\circ \psi^{-1}\) is even an isomorphism between \((L,w)^\U\) and \((K,v)^\U\).
		In particular, we have \((K,v)\equiv(L,w)\) as valued fields.
	\end{rem}
	
	\begin{rem}
		\label{rem:stable-embeddedness-is-a-trap}
		Note that even though \cite[Appendix, Lemma~1(6)]{chatzidakis1999model} might suggest otherwise, mere stable embeddedness of the residue field as a pure field is in general not enough to be able to lift automorphisms of the residue field to valued field automorphisms.
		For this, one would need to additionally assume that \(0\)-definable subsets of the residue field are traces of \(0\)-definable subsets of the valued field (which is the case in \cite[Appendix]{chatzidakis1999model}).
		This last property is also referred to as \emph{canonical embeddedness} of the residue field, which in combination with stable embeddedness amounts to \emph{full embeddedness}, see \cite[Definition~2.1.9]{cherlin-hrushovski2003finite}.
		
		This is also the case in \cite[Lemma~6.5]{anscombe2024characterizing}: 
		In their proof, stable embeddedness is not enough to lift the isomorphism of the residue fields.
		Instead one needs full embeddedness of the residue field in the sense of \cite[Definition~2.1.9]{cherlin-hrushovski2003finite}.
		In \cref{ex:fully-tame-CAKE-counterexample} we will see a tame mixed-characteristic counterexample to \CAKEeq{}, even though the residue field is stably embedded there (this is a consequence of the relative embedding property for tame fields, proved in \cite{kuhlmann2016algebra}, see  \cite[Lemma~3.1]{jahnke-simon2020nip}).
	\end{rem}

	\subsection{Counterexamples for \CAKEeq{}}
	
	We now want to look for counterexamples of \CAKEeq{}.
	With \cref{rem:wrong-proof-CAKE} in mind, we need to find examples where an automorphism of the residue field does not lift to an automorphism of the valued field.
	We start by restating a construction from \cite{anscombe-jahnke2022cohen}, then present a variation and construct counterexamples for \CAKEeq{} from there.
	
	\begin{constr}[a counterexample for lifting automorphisms, see {\cite[Example~11.5]{anscombe-jahnke2022cohen}}]
		\label{ex:hiding-example}
		Consider a perfect field \(k\) of characteristic \(p>2\) with elements \(\alpha_1,\alpha_2\in k\) such that 
		\begin{itemize}
			\item there is an automorphism \(\phi\) of \(k\) which maps \(\alpha_1\) to \(\alpha_2\), and
			\item \(\alpha_1\) and \(\alpha_2\) lie in different multiplicative cosets of \(k\x{}^2\).
		\end{itemize} 
		For example one can take \(k=\F_p(\alpha_1,\alpha_2)^\mathrm{perf}\) with \(\alpha_1,\alpha_2\) algebraically independent over \(\F_p\). We will also present more elaborate examples later.
		
		Let \((W(k),v_p)\) be the fraction field of the ring of Witt vectors over the (perfect) field \(k\) together with the Witt valuation \(v_p\).
		Let \(a_1\in \O_{v_p}\) be a representative of \(\alpha_1\) with respect to \(\res_{v_p}\), i.e. \(\res_{v_p}(a_1)=\alpha_1\).
		Take \(K\coloneqq W(k)(\sqrt{pa_1})\) and denote the (unique) extension of the Witt valuation again by \(v_p\).  
		Note that \((K|W(k),v_p)\) is purely ramified, so \(Kv_p=k\).
		
		Then, the automorphism \(\phi\) of \(k\) sending \(\alpha_1\) to \(\alpha_2\) does not lift to an automorphism of \((K,v_p)\).
	\end{constr}
	
	There is a variant of this construction where the valued field is not finitely ramified but tame of mixed characteristic.
	
	\begin{constr}[Kartas, tame version of {\cite[Example~11.5]{anscombe-jahnke2022cohen}}]
		\label{ex:tame-hiding}
		Consider a perfect field \(k\) with elements \(\alpha_1,\alpha_2\in k\) and a field automorphism \(\phi\) as in \cref{ex:hiding-example}.
		Let \((W(k),v_p)\) and \(a_1\in \O_{v_p}\) be as in \cref{ex:hiding-example}.
		Take \((K,v)\) to be an algebraically maximal immediate extension of
		\((W(k)(\sqrt{pa_1})(p^{1/p^\infty}),v_p)\)
		where \(v_p\) denotes the (unique) extension of the Witt valuation.
		Note that, again, \((W(k)(\sqrt{pa_1})(p^{1/p^\infty})|W(k),v_p)\) is purely ramified, and thus \(Kv=k\).
		Moreover, note that \((K,v)\) is tame: it is algebraically maximal, \(Kv=k\) is perfect and \(vK\) is \(p\)-divisible.
		
		The automorphism \(\phi\) of \(k\) sending \(\alpha_1\) to \(\alpha_2\) does not lift to an automorphism of \((K,v_p)\).
		The argument for this is exactly the same as in \cite[Example~11.5]{anscombe-jahnke2022cohen}.
	\end{constr}
	
	Both constructions now indeed yield counterexamples for \CAKEeq{} once we have a suitable valuation on \(k\):
	
	\begin{lem}
		\label{lem:compAKE-ctex}
		Let \(k\) be a perfect field of characteristic \(p>2\) and let \(\alpha_1,\alpha_2\in k\) such that
		\begin{itemize}
			\item there is an automorphism \(\phi\) of \(k\) sending \(\alpha_1\) to \(\alpha_2\), and
			\item \(\alpha_1\) and \(\alpha_2\) lie in different multiplicative cosets of \(k\x{}^2\). 
		\end{itemize}
		
		Let \(W(k)\) be the fraction field of the ring of Witt vectors over \(k\) and let \(a_1\in \O_{v_p}\) be a lift of \(\alpha_1\), i.e. \(\res_v(a_1)=\alpha_1\).
		Take \((K,v)\) to be 
		\begin{enumerate}[label=\normalfont(\alph*)]
			\item \(W(k)(\sqrt{pa_1})\) with the (unique) extension of the Witt vector valuation, or
			\item an algebraically maximal immediate extension of \((W(k)(\sqrt{pa_1})(p^{1/p^\infty}),v_p)\), where \(v_p\) denotes the (unique) extension of the Witt vector valuation.
		\end{enumerate} 
		
		Let \(\nu_1\) be a valuation of \(k\) such that 
		\begin{itemize}
			\item \(\nu_1(\alpha_1)\) is not divisible by \(2\) in \(\nu_1 k\), and 
			\item \(\nu_1(\alpha_2)=0\).
		\end{itemize}
		Let \(\nu_2=\nu_1\circ\phi\).
		
		Then, \((K,v)=(K,v)\) and \((k,\nu_1)\cong (k,\nu_2)\), but \((K,\nu_1\circ v)\not\equiv (K,\nu_2\circ v)\).
	\end{lem}
	\begin{proof}
		Let \(v_1=\nu_1\circ v\) and \(v_2=\nu_2\circ v\).
		We show that \((v_1K,v_1(p))\not\equiv (v_2K,v_2(p))\), which implies \((K,\nu_1\circ v)\not\equiv (K,\nu_2\circ v)\).
		We will argue that \(v_2(p)\) is divisble by \(2\) in \(v_2K\), while \(v_1(p)\in v_1K\) is not.
		
		Indeed, \(pa_1\) is a square in \(K\) and \(v_2(a_1)=\nu_2(\alpha_1)=\nu_1(\phi(\alpha_1))=\nu_1(\alpha_2)=0\), but \(v_1(a_1)=\nu_1(\res_v(a_1))=\nu_1(\alpha_1)\) is not divisible by \(2\) (note that \(\nu_1k\subseteq v_1K\) is a convex subgroup).
		Thus, \(v_2(p)=v_2(pa_1)\) is divisible by \(2\), while \(v_1(p)=v_1(pa_1)-v_1(a_1)\) is not.
		Hence \((v_1K,v_1(p))\not\equiv (v_2K,v_2(p))\).
		
		Thus, the \(\Lval\)-sentence 
		\[
		\exists X (\exists Y (Y^2=p\cdot X)\wedge v(X)=0)
		\]
		holds true in \((K,v_2)\) while it does not hold in \((K,v_1)\). 
	\end{proof}
	
	As we have already noted, for both examples we can take \(k=\F_p(\alpha_1,\alpha_2)^\mathrm{perf}\) with \(\alpha_1,\alpha_2\) algebraically independent over \(\F_p\).
	This will supply us with the following concrete counterexample for \CAKEeq{}. 
	Note that in this example the valuations on \(k\) are not henselian.
	\begin{examp}
		\label{ex:first-CAKE-counterex}
		Let \(k=\F_p(\alpha_1,\alpha_2)^\mathrm{perf}\), \(p>2\) prime with \(\alpha_1,\alpha_2\) algebraically independent over \(\F_p\).
		Denote \(t\coloneqq \alpha_1\) and \(s\coloneqq \alpha_2\), so \(k=\F_p(t,s)^\mathrm{perf}\).
		
		Let \(v_t\) be the \(t\)-adic valuation and let \(v_s\) be the \(s\)-adic valuation on \(\F_p(t,s)^\mathrm{perf}\). 
		Consider \(W(k)\), the fraction field of the ring of Witt vectors over \(k\) together with the Witt valuation \(v_p\).
		Take \(\tau\in W(k)\) such that \(\res_{v_p}(\tau)=t\).
		Let \((K,v)\coloneqq (W(k)(\sqrt{p\tau}),v_p)\).
		
		Then \((K,v)=(K,v)\) and \((k,v_s)\cong (k,v_t)\), but \((K,v_s\circ v)\not \equiv (K,v_t\circ v)\).
	\end{examp}
	
	Finally, with a little more care with regard to the valuations on \(k\) (namely we want them to be tame), we can cook up a counterexample for \AKEeq{} in mixed characteristic tame valued fields with the \textbf{same underlying field}.
	Note that, because of the F. K. Schmidt theorem, 
	it is impossible to find two independent henselian valuations on a field that is not separably closed.
	Thus, if we try to make both valuations in the example \((\F_p(t,s)^{\mathrm{perf}},v_t,v_s)\) from before henselian, we will necessarily end up with a separably closed field, and then \(s\) and \(t\) cannot be in different multiplicative cosets of the squares anymore (\(p\neq 2\)). 
	
	Instead, we construct a counterexamples where the valuations are comparable:

	\begin{constr}[tame instance of \(k\) in {\cite[Example~11.5]{anscombe-jahnke2022cohen}}]
		Let \(\Gamma\coloneqq \bigoplus_\Z \frac{1}{p^\infty}\Z\) be the lexicographic sum
		and let \(k\coloneqq \F_p(\!(\Gamma)\!)\) with the Hahn series valuation \(\nu\), \(p>2\) prime. 
		
		The valuation \(\nu\) is tame:
		\begin{itemize}
			\item the Hahn series valuation is maximal, in particular algebraically maximal,
			\item the residue field \(\F_p\) is perfect, and
			\item the value group \(\Gamma=\bigoplus_\Z \frac{1}{p^\infty}\Z\) is \(p\)-divisible.
		\end{itemize}
		
		For \(i\in \Z\), let \(e_i\coloneqq (\delta_{ij})_{j\in\Z}\in \Gamma\), i.e., \(e_i\) is the sequence that is \(0\) everywhere but in the \(i\)-th spot, where it is \(1\).
		Take \(\alpha_1\coloneqq t^{e_1}\in k\) and \(\alpha_2\coloneqq t^{e_2}\in k\).
		
		Clearly, the map \(t^{e_i}\mapsto t^{e_{i+1}}\) extends to a field automorphism of \(k\) that maps \(\alpha_1=t^{e_1}\) to \(\alpha_2=t^{e_2}\).
		We have \(\nu(\alpha_1)=e_1>e_2=\nu(\alpha_2)\) and \(e_1\) and \(e_2\) are in different archimedean classes of \(\Gamma\).
		Moreover, \(\nu(\alpha_1/\alpha_2)=e_1-e_2\) is not divisible by \(2\) in \(\Gamma\), 
		so \(\alpha_1/\alpha_2\) cannot be a square in \(k\). 
		Hence, \(\alpha_1\) and \(\alpha_2\) lie in different multiplicative cosets of the squares \(k^{\times2}\).
		
		Let \(\nu_1\) be the coarsening corresponding to quotienting out \(\Delta_1\), the biggest convex subgroup of \(\Gamma\) not containing \(e_1\). 
		This is the coarsest coarsening of \(\nu\) such that \(\alpha_1\) has strictly positive valuation.
		
		As \(e_2<e_1\) and \(e_1\) and \(e_2\) are in different archimedean classes, we have that \(\nu(\alpha_2)=e_2\in \Delta_1\), so \(\nu_1(\alpha_2)=0\). 
		Moreover, \(\nu_1(\alpha_1)=(\ldots,0,0,0,1)\in \bigoplus_{\omega^\ast}\frac{1}{p^\infty}\Z\) is not divisible by \(2\) (note \(p\neq 2\)). 
	\end{constr}
	
	To summarize, we found a witness for the following:
	
	\begin{kor}
		\label{ex:tame-Hahn-hiding-residue-field}
		There is a field \(k\) of characteristic \(p>2\) with a tame valuation \(\nu\) and elements \(\alpha_1,\alpha_2\in k\) such that
		\begin{enumerate}
			\item there is an automorphism \(\phi\) of \(k\) which maps \(\alpha_1\) to \(\alpha_2\), 
			\item \(\alpha_1\) and \(\alpha_2\) lie in different multiplicative cosets of \(k\x{}^2\), and 
			\item \(\nu(\alpha_1)\gg\nu(\alpha_2)\), i.e. \(\nu(\alpha_1)>\nu(\alpha_2)\) and they lie in different archimedean classes of \(vK\).
			\item there is a coarsening \(\nu_1\) of \(\nu\) such that \(\nu_1(\alpha_1)\) is not divisible by \(2\) and \(\nu_1(\alpha_2)=0\).
		\end{enumerate}
	\end{kor}
	
	In particular, \(\nu_1\) is a tame valuation on \(k\) satisfying the conditions of \cref{lem:compAKE-ctex}.
	Hence, we get the following example, where now \textbf{all} valuations are tame.
	
	\begin{examp}
		\label{ex:fully-tame-CAKE-counterexample}
		Take \(p>2\) prime. 
		Let \(\Gamma\coloneqq \bigoplus_\Z \frac{1}{p^\infty}\Z\) and let \(k\coloneqq \F_p(\!(\Gamma)\!)\) with the Hahn series valuation \(\nu\).
		For \(i\in \Z\) let \(e_i\coloneqq (\delta_{ij})_{j\in\Z} \in \Gamma\), i.e. \(e_i\) is the sequence that is \(0\) everywhere but in the \(i\)-th spot, where it is \(1\).
		
		For \(i=1,2\), let \(\nu_i\) be the coarsening corresponding to quotienting out \(\Delta_i\), the biggest convex subgroup of \(\Gamma\) not containing \(e_i\).
		This is the coarsest coarsening of \(\nu\) such that \(\alpha_i=t^{e_i}\) has strictly positive valuation.
		Note that \(\nu_1\) and \(\nu_2\) are tame valuations.
		
		We have that \((k,\nu_1)\cong (k,\nu_2)\) via the automorphism given by \(t^{e_i}\mapsto t^{e_{i+1}}\).
		
		Take \((K,v)\) as in \cref{ex:tame-hiding}, meaning it is an algebraically maximal immediate extension of \((W(k)(\sqrt{pa_1})(p^{1/p^\infty}), v_p)\), with \(a_1 \in W(k)\) lifting \(\alpha_1\). 
		Note that \((K,v)\) is tame.
		This now provides a counterexample to \CAKEeq{} by \cref{lem:compAKE-ctex}.
	\end{examp}
	
	\subsection{Tame counterexamples to \AKEeq{} with the same underlying field}
	
	As a consequence we can now obtain new counterexamples to \AKEeq{} with the same underlying field.
	For this, we need to recall how residue field and value group behave in compositions of valuations. Let \(K\) be a field and \(v=\barv\circ v'\) be a valuation on \(K\) with given decomposition. 
	\begin{itemize}
		\item The residue field of \((K,v)\) is \(Kv=(Kv')\barv\), the same as the residue field of the induced valuation \(\barv\).
		\item The value group \(\barv(Kv')\) of the induced valuation \(\barv\) is a convex subgroup of the value group \(vK\) of \((K,v)\) and the quotient is \(v'K=vK/\barv(Kv')\), the value group of the coarsening \(v'\). We have a short exact sequence of ordered abelian groups
		\[
		0 \to \barv(Kv') \to vK \to v'K \to 0.
		\]
	\end{itemize}
	It is easier to understand the value group \(vK\) in terms of \(v'K\) and \(\barv(Kv')\) when the short exact sequence splits.
	This can be achieved by passing to the saturated setting.
	
	We recall the following well known fact about ordered abelian groups, see e.g. \cite[Exercise~2.33]{marker2018model}. 
	It follows from \cite[Corollary~3.3.38]{aschenbrenner-dries-hoeven2017asymptotic}, as convex subgroups of ordered abelian groups are always pure.
	
	\begin{fact}
		\label{fact:split-exact-sequence-sat}
		Let \(\Gamma\) be an ordered abelian group and let \(\Delta\leq \Gamma\) be a convex subgroup.
		Moving, if needed, to an \(\aleph_1\)-saturated extension \((\Gamma^\ast,\Delta^\ast)\succeq (\Gamma,\Delta)\), the short exact sequence
		\[
		0\to \Delta^\ast \to \Gamma^\ast \to \Gamma^\ast/\Delta^\ast \to 0
		\]
		splits; that is, \(\Gamma^\ast\cong \Gamma^\ast/\Delta^\ast \oplus_\mathrm{lex} \Delta^\ast\).
	\end{fact}
	
	As a consequence we learn about the behavior of the value groups in compositions of valuations.
	
	\begin{lem}
		\label{lem:comp-value-group}
		Let \((K,v)\) and \((L,w)\) be valued fields with decompositions \(v=\barv\circ v'\) and \(w=\barw\circ w'\).
		If \(v'K\equiv w'L\) and \(\barv(Kv')\equiv \barw(Lw')\), then \(vK\equiv wL\).
	\end{lem}
	\begin{proof}
		By the Keisler--Shelah isomorphism theorem  \cite[Theorem~2.5.36]{marker2006model} (possibly invoking it several times) there is an ultrafilter \(\U\) on some index set, such that we get isomorphisms \[(v'K)^\U\overset{\cong}{\longrightarrow} (w'L)^\U\text{ and }(\barv(Kv'))^\U\overset{\cong}{\longrightarrow} (\barw(Lw'))^\U.\]
		We may also assume that \((K,v)^\U\) and \((L,w)^\U\) are \(\aleph_1\)-saturated, thus by \cref{fact:split-exact-sequence-sat}, \[vK\equiv(vK)^\U\cong (v'K)^\U\oplus_\mathrm{lex} (\barv(Kv'))^\U \cong (w'L)^\U\oplus_\mathrm{lex} (\barw(Lw'))^\U \cong (wL)^\U\equiv wL.\]
	\end{proof}

	\begin{examp}
		\label{ex:no-mix-ake}
		Let \((k,\nu_1)\), \((k,\nu_2)\) and \((K,v)\) be as in \cref{ex:fully-tame-CAKE-counterexample}. 
		That is, they are all tame fields such that \(Kv=k\), \((k,\nu_1)\cong(k,\nu_2)\) and \((K,\nu_1\circ v)\not\equiv (K,\nu_2\circ v)\).
		Since \(\chara(k)>0\), the compositions \(v_1\coloneqq \nu_1\circ v\) and \(v_2\coloneqq \nu_2\circ v\) are also tame.
		We have \(Kv_1=k\nu_1\cong k\nu_2=Kv_2\) and as \(vK=vK\) and \(\nu_1k\cong \nu_2k\), it follows from \cref{lem:comp-value-group} that \(v_1K\equiv v_2K\).
		
		In summary, we found a field \(K\), bearing two tame valuations of mixed characteristic \(v_1\) and \(v_2\) such that
		\begin{itemize}
			\item \(Kv_1\cong Kv_2\),
			\item \(v_1K\equiv v_2K\), but
			\item \((K,v_1)\not\equiv (K,v_2)\).
		\end{itemize}
	\end{examp}
	
	\begin{rem}
		\label{rem:pointed-val-gp}
		As we observed in the proof of \cref{lem:compAKE-ctex}, which applied to all of our examples, the failure of \AKEeq{} can be explained by considering the pointed value groups: we have \((v_1K,v_1(p))\not\equiv (v_2K,v_2(p))\).
		
		Philip Dittmann constructed other examples, with elementarily equivalent pointed value groups and isomorphic algebraic parts, see \cref{ex:non-ake} in the appendix.
	\end{rem}

	\subsection*{Acknowledgments}
	I would like to thank Sylvy Anscombe, Blaise Boissonneau, Anna De Mase, Philip Dittmann, Franziska Jahnke and Alexander Ziegler for helpful discussions and remarks on previous versions of this manuscript.
	I extend my thanks to Konstantinos Kartas, who pointed out that \cite[Example~11.5]{anscombe-jahnke2022cohen} can be adapted to the tame setting, see \cref{ex:tame-hiding}.
	
	I am funded by the Deutsche Forschungsgemeinschaft (DFG, German Research Foundation) under Germany's Excellence Strategy – EXC 2044/2 – 390685587, Mathematics Münster: Dynamics–Geometry–Structure. 
	Part of this paper was written while I was a guest at the Hausdorff Research Institute for Mathematics (HIM) during the Trimester Program \enquote{Definability, decidability, and computability} funded by the Deutsche Forschungsgemeinschaft (DFG, German Research Foundation) under Germany‘s Excellence Strategy – EXC-2047/1 – 390685813.
	
	\appendix
	\section{Counterexamples to AKE principles for elementary equivalence in mixed characteristic ({\normalfont by} Philip Dittmann)}
	\label{appendix}
	
	While tame valued fields generally admit a good model-theoretic treatment, see \cite{kuhlmann2016algebra},
	it is known that the general Ax--Kochen--Ershov principle for elementary equivalence fails in mixed characteristic:
	that is, there are tame valued fields $(K,v)$ and $(L,w)$ of characteristic $0$, with residue fields $Kv$ and $Lw$ of characteristic $p>0$,
	such that we have $vK \equiv wL$ as ordered abelian groups, $Kv \equiv Lw$ as fields, but $(K,v) \not\equiv (L,w)$ as valued fields.
	
	Examples for this phenomenon are given in \cite[Theorem~1.5~b),~c)]{AnscombeKuhlmann_extremal-tame},
	as well as \cref{ex:no-mix-ake} of the main text of this article (where even $K \equiv L$ as fields without a valuation).
	However, in both of these cases, the failure of $(K,v)$ and $(L,w)$ to be elementarily equivalent as valued fields
	admits very simple explanations:
	Indeed, in \cite[Theorem~1.5~b),~c)]{AnscombeKuhlmann_extremal-tame}, $K$ and $L$ have non-isomorphic algebraic parts
	(i.e.\ they disagree on which one-variable polynomials over $\mathbb{Q}$ have roots),
	while in \cref{ex:no-mix-ake}, the pointed value groups $(vK, v(p))$ and $(wL, w(p))$ are not elementarily equivalent (in the language
	of ordered abelian groups with an added constant), see \cref{rem:pointed-val-gp}.
	
	The goal of this appendix is to give examples for failures of the Ax--Kochen--Ershov principle of a different flavour,
	by constructing tame mixed-characteristic valued fields $(K,v)$ and $(L,w)$ such that $(vK, v(p)) \equiv (wL, w(p))$,
	$Kv \equiv Lw$, $K$ and $L$ have isomorphic algebraic parts, but nonetheless $(K,v) \not\equiv (L,w)$.
	
	\begin{examp}[{After \cite[Example~2.4]{ADJ_finitely-ramified} for finitely ramified fields}]\label{ex:non-ake}
		Let $p$ be an odd prime.
		Let $(K_0, v_0)$ be a maximal purely wild algebraic extension of $\mathbb{Q}$ with the $p$-adic valuation,
		so $(K_0, v_0)$ is tame and we have $K_0 v_0 = \mathbb{F}_p$ and $v_0 K_0 = \frac{1}{p^\infty}\mathbb{Z}$
		\cite[Proposition~4.5~(i), Theorem~2.1~(i)]{KuhlmannPankRoquette}.
		First taking a Gauß extension of $(K_0, v_0)$ and then passing again to a maximal purely wild algebraic extension,
		we obtain a valued field $(K_1, v_1)$ which is tame and satisfies $K_1 v_1 = \mathbb{F}_p(t)^{\mathrm{perf}}$ and $v_1 K_1 = \frac{1}{p^\infty} \mathbb{Z}$.
		Note that $K_0$ is relatively algebraically closed in $K_1$, since any finite extension of $K_0$ must
		strictly extend the residue field or the value group by tameness of $(K_0,v_0)$,
		but $K_0 v_0$ is relatively algebraically closed in $K_1 v_1$ and $v_0 K_0 = v_1 K_1$.
		Let $s \in K_1$ be an element with residue $t$.
		Let $K = K_1(\sqrt{ps})$ and $v$ the unique extension of $v_1$ to $K$,
		and let $L = K_1(\sqrt{p(s^3+1)})$ and $w$ the unique extension of $v_1$ to $L$.
		Both $K$ and $L$ are ramified quadratic extensions of $K_1$ since $v_1(ps) = v_1(p(s^3+1)) = v_1(p) = 1 \in \frac{1}{p^\infty} \mathbb{Z}$ is not divisible by $2$,
		so $(vK, v(p)) = (\frac{1}{2p^\infty} \mathbb{Z}, 1) = (wL, w(p))$,
		and $Kv = K_1 v_1 = Lw$.
		
		We claim that $(K,v) \not\equiv (L,w)$.
		Let $(K',v')$ be $K(\sqrt{p})$ with the unique extension of the valuation $v$,
		and likewise let $(L',w')$ be $L(\sqrt{p})$ with the unique extension of the valuation $w$.
		Since $\sqrt{s} \in K'$ and $\sqrt{s^3+1} \in L'$, we have $\sqrt{t} \in K' v'$ and $\sqrt{t^3+1} \in L' w'$.
		Since both $(K',v')$ and $(L',w')$ are ramified degree four extensions of $(K_1,v_1)$,
		by the fundamental equality their residue field extensions must have degree $2$ over $K_0 v_0$,
		and so we must have $K' v' = \mathbb{F}_p(t)^{\mathrm{perf}}(\sqrt{t})$ and $L' w' = \mathbb{F}_p(t)^{\mathrm{perf}}(\sqrt{t^3+1})$.
		These residue fields are not elementarily equivalent as pure fields, since the equation $Y^2 = X^3 + 1$ has different numbers of solutions:
		in $L' w'$, we have the solution $(t, \sqrt{t^3+1})$,
		but in $K' v'$, all solutions have coordinates in $\mathbb{F}_p$ since $K' v'$ is a direct limit of rational function fields over $\mathbb{F}_p$
		and the algebraic curve described by $Y^2 = X^3 + 1$ has genus $1$,
		see \cite[Example~2.4]{ADJ_finitely-ramified}.
		
		Therefore $(K',v') \not\equiv (L',w')$.
		The valued field $(K',v')$ is interpretable in $(K,v)$,
		since $K' = K(\sqrt{p})$ and the valuation ring $\mathcal{O}_{v'}$ is the integral closure of $\mathcal{O}_v$ in $K'$,
		i.e.\ the set of elements in $K'$ which are roots of a polynomial $X^2 + aX + b$ with $a, b \in \mathcal{O}_v$.
		Since $(L',w')$ is interpretable in $(L,w)$ with an interpretation defined in the same way,
		we deduce $(K,v) \not\equiv (L,w)$.
		
		Lastly, we claim that both $K$ and $L$ have the same algebraic part, namely $K_0$.
		We give the argument for $K$.
		Recall that $K_0$ is relatively algebraically closed in $K_1$.
		If $K_0$ is not relatively algebraically closed in $K$, then $K$ contains a quadratic extension of $K_0$,
		and so $K'$ above contains a degree $4$ extension of $K_0$, necessarily of ramification index precisely $2$.
		Since $(K_0,v_0)$ is tame and therefore defectless,
		it follows that $K' v'$ contains a quadratic extension of $K_0 v_0 = \mathbb{F}_p$, but this is evidently not the case.
		The argument for $L$ is analogous.
	\end{examp}
	
	\begin{rem}\label{rem:ex-amk}
		With a suitable modification of the argument above, one can construct an example of a failure
		of the Ax--Kochen--Ershov principle for elementary equivalence where $(K,v)$ and $(L,w)$ are not merely tame
		but in fact algebraically maximal Kaplansky fields,
		i.e.\ additionally $Kv$ and $Lw$ have no finite extensions of degree divisible by $p$.
		The only difficulty here is finding a suitable candidate residue field $K_1 v_1$ of characteristic $p$
		which has two elementarily non-equivalent quadratic extensions with the same algebraic part (over the prime field $\mathbb{F}_p$).
		To achieve this, one can for instance take $K_1 v_1$ to be a perfect pseudo-algebraically closed field containing the algebraic closure of
		$\mathbb{F}_p$ with two quadratic extensions distinguishable by their absolute Galois groups.
		We omit the details.
		
		Note the class of algebraically maximal Kaplansky fields admits a good quantifier elimination result \cite[Theorem~2.6]{Kuhlmann_RV}
		(there phrased up to ``amc-structures of level 0'',
		which are equivalent to the also commonly used leading term sorts RV \cite[Definition~2.1]{Flenner}).
		In this sense, this class of valued fields is as well-behaved as the class
		of henselian valued fields of residue characteristic $0$.
		The failure of the Ax--Kochen--Ershov principle for elementary equivalence
		is essentially due to the availability of the ``extra constant'' $p$,
		which has no analogue in the case of residue characteristic $0$.
	\end{rem}
	
	\begin{rem}\label{rem:alt-explanation}
		There is a different way of explaining \cref{ex:non-ake},
		in line with the approach of \cite{ADJ_finitely-ramified},
		see in particular Remark~4.8~(2) there.
		For the valued field $(K,v)$ considered in the example,
		the set
		\[ \Omega_2(K,v) := \{ \alpha \in (Kv)^\times \colon \exists a \in \mathcal{O}_v \text{ s.t.\ $a^2/p$ lies in $\mathcal{O}_v$ and has residue $\alpha$} \} \]
		is the square class of $t$ in $Kv = \mathbb{F}_p(t)^{\mathrm{perf}} =: F$,
		whereas for $(L,w)$ the corresponding set $\Omega_2(L,w)$ is the square class of $t^3+1$.
		The enriched fields $(Kv, \Omega_2(K,v)) = (F, t (F^\times)^2)$ and $(Lw, \Omega_2(L,w)) = (F, (t^3+1)(F^\times)^2)$
		(in the language of rings with a unary predicate)
		are not elementarily equivalent,
		because the elementarily non-equivalent fields $K'v'$ and $L'w'$
		are interpretable in a straightforward way.
		
		This already shows that $(K,v) \not\equiv (L,w)$,
		since $(Kv, \Omega_2(K,v))$ and $(Lw, \Omega_2(L,w))$ are interpretable in the corresponding
		valued fields by the same formulas.
		Furthermore, we even obtain a statement for the RV sorts
		of $(K,v)$ and $(L,w)$,
		given as the monoids $K/(1+\mathfrak{m}_v)$ and $L/(1+\mathfrak{m}_w)$ in a suitably enriched language,
		see \cite[Definition~2.1]{Flenner} for details.
		Indeed, $(Kv, \Omega_2(K,v))$ and $(Lw, \Omega_2(L,w))$ are straightforwardly interpretable
		in $(\mathrm{RV}_K, p)$ and $(\mathrm{RV}_L, p)$,
		where $p$ is a constant standing for the element of the RV sort induced by the field element $p$.
		Therefore $(\mathrm{RV}_K, p) \not\equiv (\mathrm{RV}_L, p)$.
		
		The sets $\Omega_2$ above,
		and their relatives $\Omega_e$ obtained by replacing squares by $e$-th powers
		for a natural number $e$ coprime to $p$,
		occur naturally in investigations of the structures $(\mathrm{RV}_K,p)$
		and $(\mathrm{RV}_L,p)$,
		compare the quantifier elimination results in
		\cite[Section~4.1 and Section~5.4]{ACGZ_distality}.
	\end{rem}

\end{document}